\documentclass{amsart}
\usepackage{amssymb,amscd,verbatim, amsthm,amsmath,amsgen,xspace, mathrsfs}
\usepackage{latexsym}
\usepackage{amsfonts}
\usepackage{color}
\usepackage{ulem}
\usepackage[all]{xy}

\def\dim{\mathop{\hbox {dim}}\nolimits}

\def\End{\mathop{\hbox {End}}\nolimits}

\def\id{\mathop{\hbox{id}}\nolimits}

\newcommand{\pf}{\begin{proof}}
	\newcommand{\epf}{\end{proof}}
\newcommand{\eq}{\begin{equation}}
	\newcommand{\eeq}{\end{equation}}
\newcommand{\eqn}{\begin{equation*}}
	\newcommand{\eeqn}{\end{equation*}}

\newcommand{\frg}{\mathfrak{g}}
\newcommand{\frh}{\mathfrak{h}}

\newcommand{\frp}{\mathfrak{p}}

\newcommand{\frs}{\mathfrak{s}}

\newtheorem{thm}[equation]{Theorem}
\newtheorem{cor}[equation]{Corollary}
\newtheorem{lemma}[equation]{Lemma}
\newtheorem{prop}[equation]{Proposition}
\newtheorem{defi}[equation]{Definition}

\newtheorem{rmk}[equation]{Remark}

\numberwithin{equation}{section}

\let\ssize\scriptstyle
\newif\ifFIRST\newdimen\MAXright\MAXright0pt
\def\sdynkin{\bgroup\eightpoint\dynkin}
\def\endsdynkin{\enddynkin\egroup}
\def\dynkin{\bgroup\FIRSTtrue\hskip.5em\setbox1\hbox{$\diagup$}%
	\setbox2\hbox{$\diagdown$}%
	\setbox0\hbox to2\wd1{\hrulefill}%
	\setbox3\hbox{$\bullet$}%
	\setbox4\hbox{$\times$}%
	\setbox7\hbox{$\circ$}
	\def\whiteroot##1{\ifFIRST\setbox5\hbox{$##1$}\ifdim\wd5>1.3em
		\hskip-.5em\hskip.5\wd5\fi\fi\FIRSTfalse
		\hskip-.25em\raise1.5\wd3\hbox to0pt{\hss\hskip.45em$
			\ssize##1$\hss}\copy7\hskip-.25em\setbox6\hbox{$##1$}
		\MAXright\wd6}
	\def\root##1{\ifFIRST\setbox5\hbox{$##1$}\ifdim\wd5>1.3em%
		\hskip-.5em\hskip.5\wd5\fi\fi\FIRSTfalse%
		\hskip-.25em\raise1.5\wd3\hbox to0pt{\hss\hskip.45em$%
			\ssize##1$\hss}\copy3\hskip-.25em\setbox6\hbox{$##1$}%
		\MAXright\wd6}%
	\def\whitedroot##1{\ifFIRST\setbox5\hbox{$##1$}\ifdim\wd5>1.3em
		\hskip-.5em\hskip.5\wd5\fi\fi\FIRSTfalse
		\hskip-.25em\lower1.8\wd3\hbox to0pt{\hss\hskip.45em$
			\ssize##1$\hss}\copy7\hskip-.25em\setbox6\hbox{$##1$}
		\MAXright\wd6}%
	\def\whiterroot##1{\hskip-.25em\copy7\hbox to0pt{\hskip.3em$\ssize##1$\hss}%
		\hskip-.25em\setbox6\hbox{\hskip.6em$##1##1$}%
		\MAXright\wd6}%
	\def\droot##1{\ifFIRST\setbox5\hbox{$##1$}\ifdim\wd5>1.3em%
		\hskip-.5em\hskip.5\wd5\fi\fi\FIRSTfalse%
		\hskip-.25em\lower1.8\wd3\hbox to0pt{\hss\hskip.45em$%
			\ssize##1$\hss}\copy3\hskip-.25em\setbox6\hbox{$##1$}%
		\MAXright\wd6}%
	\def\rroot##1{\hskip-.25em\copy3\hbox to0pt{\hskip.3em$\ssize##1$\hss}%
		\hskip-.25em\setbox6\hbox{\hskip.6em$##1##1$}%
		\MAXright\wd6}%
	\def\norroot##1{\hskip-.36em\copy4\hbox to0pt{\hskip.3em$\ssize##1$\hss}%
		\hskip-.48em\setbox6\hbox{\hskip.6em$##1##1$}%
		\MAXright\wd6}%
	\def\noroot##1{\ifFIRST\setbox5\hbox{$##1$}\ifdim\wd5>1.3em%
		\hskip-.5em\hskip.5\wd5\fi\fi\FIRSTfalse%
		\hskip-.36em\raise1.5\wd3\hbox to0pt{\hss\hskip.6em$%
			\ssize##1$\hss}\copy4\hskip-.38em\setbox6\hbox{$##1$}%
		\MAXright\wd6}%
	\def\nodroot##1{\ifFIRST\setbox5\hbox{$##1$}\ifdim\wd5>1.3em%
		\hskip-.5em\hskip.5\wd5\fi\fi\FIRSTfalse%
		\hskip-.36em\lower1.8\wd3\hbox to0pt{\hss\hskip.6em$%
			\ssize##1$\hss}\copy4\hskip-.38em\setbox6\hbox{$##1$}%
		\MAXright\wd6}%
	\def\nolink{\hskip\wd0}
	\def\link{\raise.22em\copy0}%
	\def\llink##1{\raise.32em\copy0\hskip-\wd0%
		\raise.12em\copy0\hskip-.5\wd0\hbox to0pt{\hss$##1$\hss}\hskip.5\wd0}%
	\def\lllink##1{\raise.22em\copy0\hskip-\wd0\raise.32em\copy0\hskip-\wd0%
		\raise.12em\copy0\hskip-.5\wd0\hbox to0pt{\hss$##1$\hss}\hskip.5\wd0}%
	\def\rootupright##1{\hbox to0pt{\raise.45em\copy1\hskip-.25em\raise1.3\ht1%
			\hbox{\copy3\hskip.3em$\ssize##1$}\hss}%
		\setbox6\hbox{\hskip.6em\copy1\copy1$##1##1$}%
		\ifdim\MAXright<\wd6\MAXright\wd6\fi}%
	\def\whiterootupright##1{\hbox to0pt{\raise.45em\copy1\hskip-.25em\raise1.3\ht1
			\hbox{\copy7\hskip.3em$\ssize##1$}\hss}
		\setbox6\hbox{\hskip.6em\copy1\copy1$##1##1$}
		\ifdim\MAXright<\wd6\MAXright\wd6\fi}
	\def\norootupright##1{\hbox to0pt{\raise.45em\copy1\hskip-.36em\raise1.3\ht1%
			\hbox{\copy4\hskip.3em$\ssize##1$}\hss}%
		\setbox6\hbox{\hskip.6em\copy1\copy1$##1##1$}%
		\ifdim\MAXright<\wd6\MAXright\wd6\fi}%
	\def\rootdownright##1{\hbox to0pt{\raise-.5em\copy2\hskip-.25em\raise-1.35\ht1%
			\hbox{\copy3\hskip.3em$\ssize##1$}\hss}\setbox6%
		\hbox{\hskip.6em\copy2\copy2$##1##1$}%
		\ifdim\MAXright<\wd6\MAXright\wd6\fi}%
	\def\whiterootdownright##1{\hbox to0pt{\raise-.5em\copy2\hskip-.25em\raise-1.35\ht1
			\hbox{\copy7\hskip.3em$\ssize##1$}\hss}\setbox6
		\hbox{\hskip.6em\copy2\copy2$##1##1$}
		\ifdim\MAXright<\wd6\MAXright\wd6\fi}
	\def\rootdown##1{\hbox to0pt{\hskip-.05em\vrule height.25em depth.65em%
			\hskip-.25em\raise-.95em\hbox{\copy3\hskip.3em$\ssize##1$}\hss}%
		\setbox6\hbox{$##1$}%
		\ifdim\MAXright<\wd6\MAXright\wd6\fi}%
	\def\whiterootdown##1{\hbox to0pt{\hskip-.05em\vrule height.25em depth.65em
			\hskip-.25em\raise-.95em\hbox{\copy7\hskip.3em$\ssize##1$}\hss}
		\setbox6\hbox{$##1$}
		\ifdim\MAXright<\wd6\MAXright\wd6\fi}
	\def\dots{\hskip.5em\cdots\hskip.5em}}%
\def\enddynkin{\ifdim\MAXright>1em\hskip.5\MAXright\else\hskip.5em\fi\egroup}%

\begin{document}

	\bigskip
	\title[Strongly Commuting Ring and The Prounet-Tarry-Escott Problem]{Strongly Commuting Ring and The Prounet-Tarry-Escott Problem}
	
	\author{Bin-Ni Sun and Yufeng Zhao}

	\address[Sun]{Department of Mathematics, Peking University,
		Beijing, China}

	\address[Zhao]{Department of Mathematics, Peking University,
		Beijing, China}
	\email{Zhaoyufeng@math.pku.edu.cn}
	
	\thanks{The research described in this paper is supported by grants No. 7100902512
		from Research Grant Council of China}
	
	\begin{abstract}
		
		In 1935, Wright conjectured that ideal solutions to the PTE problem in the Diophantine number theory should exist.
		In this paper, we prove that Wright's  conjecture holds true based on the representation theory  of the minuscule strongly commuting rings introduced by Kostant in 1975 and the complex coefficient cohomology  ring structures  of the  Grassmannian variety.
	\end{abstract}

	\maketitle
	
	\section{Introduction}\label{section intro}

	\vspace{0.5cm}
	
	The Prouhet-Tarry-Escott problem (written briefly as PTE) is an old unsolved problem in the Diophantine number theory, which has a long history and is, in some form, over 200 years old. In its most general setting, the PTE problem
	consists of finding two lists of integers $\mathbb{X}=[x_{1}, x_{2}, \dots, x_{n}]$ and $\mathbb{Y}=[y_{1}, y_{2}, \dots, y_{n}]$,  distinct up to permutation, such that
	\begin{equation} 	
		\sum_{i = 1}^{n}{x_{i}^j} = \sum_{i = 1}^{n}y_{i}^{j}, \qquad j = 1, 2, \dots, m.
	\end{equation}
	We call $n$ the size of the solution and $m$ the degree. The solution sets $\mathbb{X}$ and $\mathbb{Y}$ are usually represented as $\mathbb{X}=_m \mathbb{Y}$.
	A trivial solution of PTE means that the $x$'s merely form a permutation of the $y$'s.

	For more general history and results related to the PTE problem, we refer to the recent review by Srikanth and Veena [SV] and Chapter 11 of the book by Borwein [Bo].

	From ([A], [D]), it is known that for the non-trivial solution of equation (1.1) to exist we must have $n \geq m+1$. In this paper, we will present a representation theoretic  proof
	for this result based on the infinitesimal central character theory of the special linear Lie algebra.

	The PTE problem  with  degree $n-1$ and size $n$  is  referred to be an ideal solution.
	In 1935, Wright [W] conjectured that ideal solutions to the PTE problem should exist.
	In this paper, we will also develop the representation theory  of the strongly commuting rings introduced by Kostant [Ko1] in 1975 to claim that the conjecture holds true. 
	
	To formulate our main theorems, we need the following definitions and assertions.
	Let $\frg$ be a complex simple Lie algebra and $Z(\frg)$ be the center of its universal enveloping algebra $U(\frg)$.
	One says that the arbitrary representation $\pi: U(\frg)\rightarrow \End V_{\pi}$
	admits an infinitesimal character $\chi_{\pi}$. For every $u \in Z(\frg)$, $\pi(u)$ reduces to a scalar operator $\chi_{\pi}(u)$ on $V_{\pi}$.
	It is a theorem of Dixmier that any irreducible representation  of $U(\frg)$ admits an infinitesimal character.
	
	Fix a finite-dimensional irreducible $\frg$-module $V_\lambda$  with highest weight $\lambda$.
	Let
	$$\pi\colon U(\frg) \rightarrow \End V_\pi $$
	be an another  arbitrary  $\frg$-module having an infinitesimal central character $\chi_\pi$.
	In 1975 around, to study the infinitesimal characters occurring in the tensor product 
	$V_\lambda\otimes V_\pi$, Kostant [Ko1] introduced  the $\frg$-invariant endomorphism  algebras  $$R_\lambda(\frg)= (\End V_\lambda\otimes U(\frg))^\frg$$
	and
	$$R_{\lambda,\pi}(\frg)=(\End V_\lambda\otimes \pi[U(\frg)])^\frg,$$
	which were named the strongly commuting rings by the author.

	As  shown in [Ko2], there is a $\frg$-submodule $E$ of $U(\frg)$
	such that the multiplication
	$$ Z(\frg)\otimes E\rightarrow U(\frg)$$
	is a $\frg$-module isomorphism and  $R_\lambda(\frg)$ is a free $Z(\frg)$-module.
	It was proved that 
	$R_\lambda(\frg)$ and $ R_{\lambda,\pi}(\frg)$ are commutative if and only if all the weight spaces $V_{\lambda}$ are 1-dimensional (cf.  [Pan], [Ki1] and [Ki2]).
	The commutative $\frg$-invariant endomorphism algebras
	have many connections with equivariant cohomology [Pan].
	
	Let $\mbox{Gr}^{k}(\mathbb{C}^{n})$
	denote the  Grassmannian variety of subspaces in $\mathbb{C}^{n}$ of dimension $k$. The well known  Schubert basis of the complex coefficient cohomology  ring  $H^{*}(\mbox{Gr}^{k}(\mathbb{C}^{n}))$ of $\mbox{Gr}^{k}(\mathbb{C}^{n})$ was given in terms of Schur polynomials [F].
	
	Assume $\{\omega_k \ | \ k=1,\cdots, n-1\}$ are
	the fundamental dominant weights for  the special linear simple Lie algebra $\mbox{sl}_{n}(\mathbb{C}) $ which are minuscule,  in the sense that the Weyl group action of $\mbox{sl}_{n}(\mathbb{C}) $ on the weight set of $V_{\omega_k}$ is transitive.
	In this paper, we will present a type of matrix basis for the free  $Z(\mbox{sl}_{n}(\mathbb{C}))$- module $R_{\omega_k}(\mbox{sl}_{n}(\mathbb{C}))$, which can be viewed as a kind of matrix model for the complex coefficient cohomology  ring of the  Grassmannian variety
	$\mbox{Gr}^{k}(\mathbb{C}^{n})$.
	
	Consider the map
	$$\delta: U(\frg) \rightarrow U(\frg)\otimes U(\frg)$$
	defined by
	$$ \delta(x)= x\otimes 1 +  1\otimes x\text{ for } x \in \frg,$$
	which is extended to be  a homomorphism of associative algebras.

	Let $\delta_{\lambda} = (\pi_{\lambda} \otimes \id) \circ \delta$ and we have the map
	$$\delta_\lambda: U(\frg) \rightarrow \End V_\lambda\otimes U(\frg)$$
	defined by
	$$ \delta_\lambda(x)= \pi_{\lambda}(x)\otimes 1 + \textrm{id}\otimes x\text{ for } x \in \frg,$$
	which is also extended to be a homomorphism of associative algebras.

	Let $Z(\mathfrak{g})$ be the center of $U(\mathfrak{g})$ and $z$ be a fixed element of $  Z(\mathfrak{g}) $. Define 
	$$
	M_{\lambda}(z) = \delta_{\lambda}(z) - \pi_{\lambda}(z) \otimes 1 -  \textrm{id} \otimes z. 
	$$
	If $V_{\lambda}$ has dimension $d_{\lambda}$, then $M_{\lambda}(z)$ is  a $d_{\lambda} \times d_{\lambda}$ matrix with entries in $U(\mathfrak{g})$.

	\begin{thm}
		Let $\mathfrak{C}_2, \cdots, \mathfrak{C}_{n}$ be Casimir generators for the center of the universal enveloping algebra of the special linear simple Lie algebra  $\mbox{sl}_{n}(\mathbb{C}) $ and  $V_{\omega_{k}}$ be the fundamental modules  with $1 \leq k \leq \lceil \frac{n}{2} \rceil$. Then 
		the set of 	matrices $M_{\omega_{k}}(\mathfrak{C}_{2}), \dots, M_{\omega_{k}}(\mathfrak{C}_{k+1})$ is a minimal set of generators for 
		the endomorphism algebra $R_{\omega_{k}}(\mbox{sl}_{n}(\mathbb{C}))$ over $Z(\mbox{sl}_{n}(\mathbb{C}))$.
		
	\end{thm}
	
	\begin{thm}
		(1)	For a non-trivial solution to the PTE problem with size $n$ and degree $m$ to exist, we must have $n \geq m+1$.
		(2) An ideal solution of the PTE problem with any degree always exists.
	\end{thm}
	
	\begin{rmk}	Since  the dual module for the fundamental module $V_{\omega_{k}}$ of  $\mbox{sl}_{n}(\mathbb{C})$ is isomorphic to $V_{\omega_{n-k}}$ and
		$\text{End } V_{\omega_{k}} \simeq V_{\omega_{n-k}} \otimes V_{\omega_{k}}$, without any loss of generality we can assume $1 \leq k \leq n-k$.	
	\end{rmk}

	The paper is organized as follows. 
	In Section 2, we will first collect some necessary information on the central character theory of simple Lie algebras and give a representation theoretic proof of the first part of Theorem 1.3 of the PTE problem. 
	In Section 3, some basic properties of  Kostant's strongly commuting rings $R_\lambda(\frg)$ and $R_{\lambda,\nu}(\frg)$ will be given. 
	Section 4 is devoted to studying the structures of the Weyl group module  $R_{\omega_{k}}(\frh)$. 
	In Section 5, we will prove Theorem 1.2 by presenting the basis for the free module $R_{\omega_{k}}(\mbox{sl}_{n}(\mathbb{C}))$ over $Z(\mbox{sl}_{n}(\mathbb{C}))$.
	In Section 6, the relationship between the  basis structure of $R_{\omega_{k},\nu}(\mbox{sl}_{n}(\mathbb{C}))$  and the PTE problem will be introduced. 
	And we will finally  prove Wright's conjecture, i. e. the second part of Theorem 1.3,  holds true.

	\section{ Infinitesimal central character  and PTE problem}
	
	In this Section, we will first collect some necessary information on the central character theory of simple Lie algebra and then we will give a representation theoretic  proof of the first part of Theorem 1.3 of the PTE problem.
	
	Let $\frg$ be a complex simple Lie algebra with a Cartan subalgebra $\frh$
	of dimension $l$, $\mathfrak{h}^{*} $ be the dual space to $ \mathfrak{h} $, $\Phi$ be a system of roots and $\Phi_{+}$ be a system of positive roots.	Write
	the root space decomposition of $\frg$:
	$$\frg=\frh \oplus \mathfrak{n} \oplus \mathfrak{n}^{-},
	\mathfrak{n} = \sum\limits_{\alpha\in \Phi_{+}}\frg^{\alpha}, \mathfrak{n}^{-}=\sum\limits_{\alpha\in \Phi_{+}}\frg^{-\alpha}$$and $\rho=\frac{1}{2}\sum\limits_{\alpha\in \Phi_{+} }\alpha$.
	Assume  $\bigtriangleup =\{\alpha_{1}, \cdots,\alpha_{l}\}$ is the
	simple root system in $\Phi_{+}$ and $W$ is  the Weyl group.  
	The Killing form on $\frg$ is defined by:
	$$B(.,.):\frg\times \frg \rightarrow \mathbb{C};(x,y)\rightarrow \mbox{tr}(\mbox{ad}x\mbox{ad}y), \ \forall \ x, y \in \frg.$$
	The Killing isomorphism $$\mathscr{K} : \frg\rightarrow \frg^*; x\rightarrow B(x,.)$$is an isomorphism of $\frg$-modules. For any $\alpha, \beta \in \Phi$, denote $$(\alpha,\beta)=B(\mathscr{K}^{-1}\alpha, \mathscr{K}^{-1}\beta).$$
	Let $\mathcal{P}_{+}$ be the set of dominant integral weights for $\mathfrak{g}$. The fundamental dominant weights are denoted by $\{\omega_1,\dots, \omega_l\}$ satisfying  $\frac{2(\omega_{i},\alpha_{j})}{(\alpha_{j},\alpha_{j})}=\delta_{i,j}$.

	\begin{defi}
		Fix an irreducible  $\frg$-module $V_\lambda$  with highest weight $\lambda$. Any nontrivial element $u\in Z(\frg)$ acts on   $V_\lambda$ by a scalar multiplication. Write $$u|_{V_\lambda}=\chi_{\lambda}(u).$$
		Thus $$\chi_{\lambda}:Z(\frg) \rightarrow \mathbb{C};u \rightarrow  \chi_{\lambda}(u)$$is a 1-dimensional representation of $Z(\frg)$, which is called the infinitesimal central character of $V_\lambda$. 
	\end{defi}

	Consider the zero weight space of $U(\frg)$:
	$$U(\frg)_0=\{x\in U(\frg) \ | \ [h,x]=0, \ \forall \ h \in \frh\}.$$
	The vector space $$\mathfrak{K} =
	U(\frg)_0 \cap U(\frg)\mathfrak{n} =U(\frg)_0 \cap \mathfrak{n}^{-}U(\frg)$$ is a 2-sided ideal of $U(\frg)$
	and $U(\frg)_0= \mathfrak{K} \oplus U(\frh)$. Let $$\psi:U(\frg)_0\rightarrow U(\frh)$$ be the projection map obtained from this decomposition, which is called Harish-Chandra homomorphism.
	For any $\lambda \in \mathfrak{h}^{*}$, the infinitesimal central character is given by 
	$$\chi_{\lambda}(u)=\lambda(\psi(u)), \ \forall u \in Z(\mathfrak{g}).$$ 
	
	\begin{lemma}Let 
		$\lambda, \mu \in \frh^{*}$, then
		$\chi_{\lambda}=\chi_{\mu}$ if and only if $\mu+\rho=w(\lambda+\rho)$ for some $w \in W$ [C].
	\end{lemma}

	Now we restrict the complex simple Lie algebra $\frg$ to be  the special linear Lie algebra  $\mbox{sl}_{n}(\mathbb{C})$.
	Set 
	$$X_{ij} = E_{ij} \quad (1 \leq i \ne j \leq n),$$
	$$X_{ii} = E_{ii} - \frac{1}{n} \sum_{k = 1}^{n} E_{kk} \quad (1 \leq i \leq n),$$
	where $E_{ij}$ is an $n \times n$ matrix with the entry in row $i$ and column $j$ equal to $1$ and others equal to $0$.

	It is obvious that $\{X_{ij}, X_{kk}| 1 \leq i \ne j \leq n, 1 \leq k \leq n-1 \}$ is a set of the basis for $\mbox{sl}_{n}(\mathbb{C})$ and $[X_{ij}, X_{km}] = \delta_{jk}X_{im} - \delta_{im}X_{kj}$ for $1 \leq i, j, k, m \leq n$.

	Define $$\mathfrak{C}_k=\sum\limits_{i_1,\cdots, i_k=1}^{n}X_{i_1i_2}X_{i_2i_3}\cdots X_{i_ki_1}.$$It is known that $\{\mathfrak{C}_2, \cdots, \mathfrak{C}_{n}\}
	$ is an algebraically independent set which algebraically generates $Z(\mbox{sl}_{n}(\mathbb{C}))$, i.e. the center of the universal enveloping algebra for $\mbox{sl}_{n}(\mathbb{C})$ [Po]. 
	
	For any $\lambda \in \mathfrak{h}^{*}$, through taking the irreducible quotient of the Verma module we can get an irreducible $\mathfrak{g}$-module $V_{\lambda}$ with highest weight $\lambda$ [C, Hu]. Let  $v_{\lambda}$  be  a highest weight vector for $V_{\lambda}$. Based on the generating function techniques, an expansion of the infinitesimal central character
	$\chi_{\lambda}(\mathfrak{C}_{p})$ in the power of sums  was presented by Popov [Po]:

	\begin{prop}
		Let $\lambda \in \mathfrak{h}^{*}$ and $$X_{i,i}.v_{\lambda}=m_iv_{\lambda}, i=1, \cdots,n.$$	
		Define 
		\begin{equation}
			S_{k}(\lambda) = \sum_{i = 1}^{n} [(m_{i} + n - i)^{k} - (n - i)^{k}],
			\ k \in \mathbb{N}.
		\end{equation}
		Then\begin{equation}\chi_{\lambda}(\mathfrak{C}_{p}) = \sum_{\vec{t} = (t_{1}, t_{2}, ..., t_{k})} \beta_{p}(\vec{t})S_{1}^{t_{1}}(\lambda)S_{2}^{t_{2}}(\lambda) \cdots S_{k}^{t_{k}}(\lambda)
		\end{equation}\begin{equation}	
			=S_{p}(\lambda)+\sum_{\vec{t} = (t_{1}, t_{2}, ..., t_{k})\neq (0,\cdots,1)} \beta_{p}(\vec{t})S_{1}^{t_{1}}(\lambda)S_{2}^{t_{2}}(\lambda) \cdots S_{k}^{t_{k}}(\lambda).	
		\end{equation}
		The summation is over all sets $\vec{t}$ of non-negative integers $t_{1}, t_{2}, ..., t_{k}$ satisfying the condition $1 \leq K \leq p$, where 
		\[K = -1 + \sum_{k} (k+1)t_{k}\] 
		and the expressions of coefficients $\beta_{p}(\vec{t})$ coincide with formula (3.9) in
		[Po].
	\end{prop}
	
	\begin{lemma} (Lemma 8.3.4 in  [Wi])
		Let $\lambda, \mu \in \mathcal{P}_{+}$. If there exists $w \in W$ such that $w(\lambda) = \mu$, then $\lambda = \mu$. 
	\end{lemma}

	\begin{defi}
		The PTE problem
		consists of finding two lists of integers $\mathbb{X}=[x_{1}, x_{2}, \dots, x_{n}]$ and $\mathbb{Y}=[y_{1}, y_{2}, \dots, y_{n}]$,  distinct up to permutation, such that
		$$
		\sum_{i = 1}^{n}x_{i}^{j} = \sum_{i = 1}^{n}y_{i}^{j}, \qquad j = 1, 2, \dots, m.
		$$
		We call $n$ the size of the solution and $m$ the degree. The solution sets $\mathbb{X}$ and $\mathbb{Y}$ are usually represented as $\mathbb{X}=_m \mathbb{Y}$.
		A trivial solution of PTE means that the $x$'s merely form a permutation of the $y$'s.
		
	\end{defi}
	
	\begin{thm}
		If two  lists of integers $\mathbb{X} = [x_{1}\geq x_{2}\geq \dots \geq x_{n}]$ and $\mathbb{Y} = [y_{1}\geq y_{2}\geq \dots\geq y_{n}]$ satisfy
		$$
		\sum_{i = 1}^{n}x_{i}^{j} = \sum_{i = 1}^{n}y_{i}^{j}, \qquad j = 1, 2, \dots, n,
		$$
		then  $\mathbb{X} = \mathbb{Y}$.
	\end{thm}
	
	\begin{proof}
		Let $x_{i}'$ and $y_{i}'$ satisfy the relations as follows:
		$$(x_1',\cdots,x_n')+(n-1,n-2, \cdots,0)=(x_1, \cdots,x_n),$$
		$$(y_1',\cdots,y_n')+(n-1,n-2, \cdots,0)=(y_1, \cdots,y_n).$$
		Denote $$h_i=X_{ii}-X_{i+1,i+1},i=1,\cdots,n-1$$
		and set $$\lambda=\sum_{i = 1}^{n-1}(x_{i}'-x_{i+1}')\omega_{i},
		\mu=\sum_{i = 1}^{n-1}(y_{i}'-y_{i+1}')\omega_{i},$$ 
		where $\omega_{i} (1 \leq i \leq n-1)$ are the fundamental wights of the special linear Lie algebra $\mbox{sl}_{n}(\mathbb{C})$ satisfying $\omega_{i}(h_{j}) = \delta_{i,j}$ for $1 \leq i, j \leq n-1$. 
		For the highest weight module $V_{\lambda}$, we have 
		$$X_{11}. v_{\lambda} = \frac{(n-1)h_1+(n-2)h_2+\cdots+h_{n-1}}{n}v_{\lambda} = 
		(x_{1}'-\frac{\sum\limits_{i=1}^{n}x_{i}'}{n})v_{\lambda}.$$
		Inductively,
		$$X_{jj}. v_{\lambda} = (X_{j-1,j-1}-h_{j-1}). v_{\lambda} =	(x_{j}'-\frac{\sum\limits_{i=1}^{n}x_{i}'}{n})v_{\lambda}, j=2,\cdots,n.$$
		The equation $$
		\sum_{i = 1}^{n}x_{i}^{j} = \sum_{i = 1}^{n}y_{i}^{j}, \qquad j = 1, 2, \dots, n
		$$ and Equation (2.4) imply $$S_k(\lambda)=S_k(\mu), k=1,\cdots,n.$$
		It follows from  Popov's formula (2.5) and (2.6) that: $$\chi_{\lambda}(\mathfrak{C}_k) = \chi_{\mu}(\mathfrak{C}_k), k=2, \cdots, n.$$
		Since the set $\{\mathfrak{C}_2, \cdots, \mathfrak{C}_{n}\}
		$ algebraically  generates $Z(\mbox{sl}_{n}(\mathbb{C}))$ and the central character is an algebra homomorphism, we have $
		\chi_{\lambda}=\chi_{\mu}$. By Lemma 2.2, we have $\lambda+\rho=w(\mu+\rho)$ for some $w \in W$. Note that $\lambda+\rho=\sum_{i = 1}^{n-1}(x_{i}'-x_{i+1}'+1)\omega_{i}$ and $\mu+\rho=\sum_{i = 1}^{n-1}(y_{i}'-y_{i+1}'+1)\omega_{i}$, which are dominant integral weights. Then, by Lemma 2.7, we get $\lambda=\mu$ and from  the equation $$\lambda(X_{ii})=x_{i}'-\frac{\sum\limits_{j=1}^{n}x_{i}'}{n},
		\mu(X_{ii})=y_{i}'-\frac{\sum\limits_{j=1}^{n}y_{i}'}{n}, i=1,\cdots, n$$to obtain $\mathbb{X} = \mathbb{Y}$.
	\end{proof}

	\section{Basic Structures of  $R_\lambda(\frg)$ and $R_{\lambda,\nu}(\frg)$}

	In this section, we will introduce the strongly commuting ring  $R_\lambda(\frg)$ defined  by Kostant  [Ko1] and recall some related results studied by some other authors [B, Pan, Ki1, Ki2]. 
	
	For $\lambda \in \mathcal{P}_{+}$, we assume that the irreducible $\mathfrak{g}$-module $V_{\lambda}$ with highest weight $\lambda$ is finite-dimensional. 
	Denote $\Pi(V_{\lambda})$ the weight set of $\frg$-module	$V_{\lambda}$ and $m_{\lambda}^{\mu}$ the dimension of the $\mu$-weight space $V_{\lambda}^{\mu}$ of $V_{\lambda}$ for $\mu \in \Pi(V_{\lambda})$.

	\begin{defi} The $\mathfrak{g}$-module $V_{\lambda}$  is said to be small, if twice of any root of $\frg$ is not a weight  for $V_{\lambda}$.
	\end{defi} 
	
	\begin{defi}A finite-dimensional irreducible $\frg$-module $V_\lambda$  with highest weight $\lambda$ is said to be minuscule if $\frac{2(\lambda, \alpha)}{(\alpha, \alpha)} = 0, 1, -1$ for all roots $\alpha \in \Phi$. 
	\end{defi}
	
	The following result is given in  Theorem 1.1 of [R]:
	\begin{prop}
		Assume $\lambda$ is minuscule. Then every constituent of $\End V_{\lambda}$ is small.
	\end{prop}
	\begin{proof}
		Since $\lambda$ is minuscule, we have  $\frac{2(\lambda, \alpha)}{(\alpha, \alpha)} = 0, 1, -1$ for all roots $\alpha \in \Phi$. By the Kumar's result in [Ku], every weight in $\End V_{\lambda}$ is Weyl group $W$-conjugate to $\lambda-w\lambda$ for some $w\in W$. If $\beta\in \Phi$ and $2\beta$ is a weight in $\End V_{\lambda}$, say $2\beta=\lambda-w\lambda$ for some $w\in W$. We have $$(\lambda,\lambda)=(w\lambda,w\lambda)=(\lambda-2\beta,\lambda-2\beta),$$and then
		$\frac{2(\lambda, \beta)}{(\beta, \beta)} = 2$. This is a contradiction. Hence
		every constituent of $\End V_{\lambda}$ is small. 
	\end{proof}

	Let $\mathscr{P}(\frg)$ be the algebra of all polynomials on $\frg$ and 
	$\frh$ be a Cartan subalgebra of $\frg$. Assume
	$$\phi : \mathscr{P}(\mathfrak{g}) \rightarrow \mathscr{P}(\mathfrak{h}); p\rightarrow p| _{\frh}$$is the restriction map. We know that  $\phi : \mathscr{P}(\mathfrak{g})^{\frg} \rightarrow \mathscr{P}(\mathfrak{h})^{W}$ is an algebra isomorphism. 
	Let $S(\mathfrak{g})$ be the symmetric algebra over the dual space to $\mathfrak{g}$.
	Then the Killing isomorphism $\mathscr{K}:\frg\rightarrow \frg^{*}$ can be extended to a $\frg$-module isomorphism $$\tilde{\mathscr{K}}:S(\frg)\rightarrow \mathscr{P}(\frg)=S(\frg^{*}).$$
	Due to the ad-invariance of the Killing form, we get $\tilde{\mathscr{K}}(S(\frg)^{\frg})=\mathscr{P}(\frg)^{\frg}$.We  denote these isomorphic algebras $S(\frg)^{\frg}\simeq \mathscr{P}(\frg)^{\frg}\simeq Z(\frg)$ by $J$ in common.

	\begin{defi} For $\nu \in \mathcal{P}_{+}$, 
		the space $$J_{\nu}(\mathfrak{g})=(V_{\nu}\otimes {\mathscr{P}}(\frg))^{\frg}$$ is called the $J$-module of $\mathfrak{g}$-covariants of type $\nu$. 
	\end{defi}

	It was shown  by Kostant [Ko2] that  $J_{\nu}(\mathfrak{g})$ is a free $J$-module
	with rank $ m_{\nu}^{0}$.
	Let $\{f_1 ,\cdots , f_l\}$ be  a set of basic invariants of $\mathscr{P}(\mathfrak{g})^{\frg}$. Then
	$$J_{\nu}(\mathfrak{g})=(f_1 ,\cdots , f_l)J_{\nu}(\mathfrak{g})\oplus H_{\nu},$$where $H_{\nu}$ is a  graded finite-dimensional vector space over $\mathbb{C}$.
	Any homogeneous $\mathbb{C}$- basis for $H_{\nu}$ is also  a free basis for the free $\mathscr{P}(\mathfrak{g})^{\frg}$-module $J_{\nu}(\mathfrak{g})$.
	

	Since the  zero-weight space $V_{\nu}^{0}$ is a Weyl group $W$-module, the associative algebra $$J_{\nu}(\mathfrak{h}) = (V_{\nu}^{0} \otimes \mathscr{P}(\mathfrak{h}))^W$$ is a $\mathscr{P}(\mathfrak{h})^{W}$-module. 
	
	The elements in 
	$J_{\nu}(\mathfrak{g})$ 
	can be identified with $\frg$-equivariant morphisms from $\frg$ to $V_{\nu}$. More precisely, an element $\sum v_i\otimes p_i \in J_{\nu}(\mathfrak{g})$ defines the morphism that takes $x\in \frg$ to $\sum p_i(x)v_i \in V_{\nu}$.
	
	Restricting a $\frg$-equivariant morphism from $\frg$ to $V_{\nu}$ to $\frh$ yields a $W$-equivariant morphism from $\frh$ to $V_{\nu}^{0}$. In other words, we obtain 
	a homomorphism $\varphi_{\nu} : J_{\nu}(\mathfrak{g}) \rightarrow J_{\nu}(\mathfrak{h})$ of free graded $J$-modules with the same rank $m_{\nu}^{0}$.
	The following theorem is due to Broer [B]:

	\begin{thm} The homomorphism $\varphi_{\nu} : J_{\nu}(\mathfrak{g}) \rightarrow J_{\nu}(\mathfrak{h})$
		is a $J$-module isomorphism if and only if $\mathfrak{g}$-module $V_{\nu}$  is small.
	\end{thm}

	\begin{prop}If $\lambda$ is  minuscule, the two free $J$-modules 
		$(\End V_\lambda\otimes \mathscr{P}(\frg))^{\mathfrak{g}}$ and $(\End_{\frh}V_\lambda\otimes \mathscr{P}(\frh))^{W}$ are isomorphic.
	\end{prop}

	\begin{proof}By Theorem 1.1 in [R], every constituent of  $\End V_{\lambda} \simeq V_{\lambda} \otimes V_{\lambda}^* = \bigoplus c_{\nu}V_{\nu}$ is small. This implies that 
		the restriction morphism
		$$
		\tilde{\varphi}_{\lambda} :  (\End V_\lambda\otimes \mathscr{P}(\frg))^{\mathfrak{g}} \rightarrow (\End_{\frh}V_\lambda\otimes \mathscr{P}(\frh))^{W}
		$$
		is an isomorphism by Theorem 3.5 through patching together all the homomorphisms $\varphi_{\nu} $ in the constituents.
		
	\end{proof}

	It is well known that there is a $\frg$-submodule $E$ of $U(\frg)$ such that the multiplication
	$$ Z(\frg)\otimes E\rightarrow U(\frg)$$
	is a $\frg$-module isomorphism and the $\mathfrak{g}$-endomorphism algebra of type $\lambda$
	$$R_{\lambda}(\mathfrak{g}) = (\End V_\lambda\otimes U(\frg))^{\mathfrak{g}}$$
	is a free $Z(\frg)$-module, which was named the strongly commuting ring by Kostant [Ko1]. 
	Set 
	$$R_{\lambda}(\frh) = (\End_{\frh}V_\lambda\otimes U(\frh))^{W}.$$

	\begin{prop}If $\lambda$ is  minuscule, the two free $J$-modules
		$R_{\lambda}(\mathfrak{g})$ and $R_{\lambda}(\frh)$ are  isomorphic with the same rank  $\dim V_{\lambda}$.
	\end{prop}
	
	\begin{proof}
		Let $S_{p}$ be the symmetric group.
		Then the symmetrization map 
		$$\theta:S(\frg)\rightarrow U(\frg);y_1\cdots y_p
		\mapsto \frac{\sum\limits_{\sigma \in S_p}y_{\sigma (1)}\cdots y_{\sigma (p)}}{p!}, \ \forall \ y_i \in \frg, p \in \mathbb{N}$$ is a  $\frg$-module isomorphism. 
		We get the following $J$-module isomorphisms
		$$R_{\lambda}(\mathfrak{g}) = (\End V_\lambda\otimes U(\frg))^{\mathfrak{g}} 
		\stackrel{\mbox{id}\otimes \theta^{-1}}{\rightleftharpoons}
		(\End V_\lambda\otimes S(\frg))^{\mathfrak{g}}
		\stackrel{\mbox{id}\otimes \tilde{K}}{\rightleftharpoons}
		(\End V_\lambda\otimes \mathscr{P}(\frg))^{\mathfrak{g}} $$
		and if $\lambda$ is minuscule, we have the following $J$-module isomorphisms
		$$ (\End V_\lambda\otimes \mathscr{P}(\frg))^{\mathfrak{g}} 
		\stackrel{	\tilde{\varphi}_{\lambda}}{\rightleftharpoons}  
		(\End_{\frh}V_\lambda\otimes \mathscr{P}(\frh))^{W}   
		\stackrel{\mbox{id}\otimes \tilde{K}^{-1}}{\rightleftharpoons} 
		(\End_{\frh}V_\lambda\otimes S(\frh))^{W}
		\stackrel{\mbox{id}\otimes \theta}{\rightleftharpoons} 
		R_{\lambda}(\frh).$$
		Now we arrive that 
		$$\psi_{\lambda} = (\mbox{id}\otimes \theta) \circ 
		(\mbox{id}\otimes \tilde{K}^{-1}) \circ 
		\tilde{\varphi}_{\lambda} \circ 
		(\mbox{id}\otimes \tilde{K}) \circ 
		(\mbox{id}\otimes \theta^{-1}) : R_{\lambda}(\frg)\rightarrow R_{\lambda}(\frh)$$ 
		is a $J$-module isomorphism if $\lambda$ is minuscule.

	\end{proof}

	For any $\nu \in \mathcal{P}_{+}$, let$$\tilde{\pi}_{\nu}
	\colon \End V_\lambda\otimes U(\frg) \rightarrow   \End V_\lambda\otimes \End V_\nu $$
	be the surjective homomorphism defined by $\tilde{\pi}_{\nu}=\mbox{id}\otimes \pi_{\nu}$. Now both 
	$ \End V_\lambda\otimes U(\frg) $ and $\End V_\lambda\otimes \End V_\nu $ are completely reducible as $\frg$-modules with respect to the adjoint actions.
	
	We denote $$R_{\lambda,\nu}(\frg)=(\End V_\lambda\otimes\End V_\nu)^\frg.$$
	Since 
	$\tilde{\pi}_{\nu}$ is a surjective $\frg$-module map, it carries invariants onto invariants and then by restriction $\tilde{\pi}_{\nu}$ induces a surjective homomorphism $$\gamma_{\nu}\colon  (\End V_\lambda\otimes U(\frg))^\frg
	\rightarrow (\End V_\lambda\otimes\End V_\nu)^\frg,$$i. e. 
	$$\gamma_{\nu}(R_{\lambda}(\frg))=R_{\lambda,\nu}(\frg).$$
	
	Following Kostant [Ko1], we make the following definition.
	\begin{defi}
		We say that $\lambda$ is totally subordinate to $\nu$ if the number of irreducible constituents
		in $V_\lambda\otimes V_\nu$ is equal to $d_\lambda\colon=\dim V_\lambda$.
	\end{defi}
	If $\lambda$ is totally subordinate to $\nu$, then $\dim R_{\lambda,\nu}(\frg)$ is equal to the free rank of  $Z(\frg)$-module 
	$R_\lambda(\frg)$ [Ko1].

	\begin{prop}Assume $\lambda$ is a  minuscule weight and totally subordinate to $\nu$. If $s_i$ is a free basis for $Z(\frg)$-module 
		$R_\lambda(\frg)$, then $\gamma_{\nu}(s_i)$ is a basis of $R_{\lambda,\nu}(\frg) $ as the  $\mathbb{C}$-vector space.
		
	\end{prop}
	
	\begin{proof}
		
		Since $Z(\frg)$ maps into scalars under $\gamma_{\nu}$  and 
		$$R_{\lambda,\nu}=(\End V_\lambda\otimes\End V_\nu)^\frg$$ has dimension $d_{\lambda}$, the proposition  follows from the surjectivity of $\gamma_{\nu}$.

	\end{proof}

	Two kinds of special elements play a vital role in constructing the free basis for $R_\lambda(\frg)$.
	The one due to Kostant is defined as follows.
	Consider the map
	$$\delta_\lambda: U(\mathfrak{g}) \rightarrow \End V_\lambda\otimes U(\mathfrak{g})$$
	defined by
	$$ \delta_\lambda(x)= \pi_{\lambda}(x)\otimes 1 +  \textrm{id}\otimes x  \text{ for } x \in \mathfrak{g},$$
	which is extended to be a homomorphism of associative algebras. For  a fixed element  $z$ of $  Z(\mathfrak{g}) $, Kostant [Ko1] defined such particular element 
	\begin{equation}
		M_{\lambda}(z) = \delta_{\lambda}(z) - \pi_{\lambda}(z) \otimes 1 -  \textrm{id} \otimes z \in R_\lambda(\frg). 
	\end{equation}
	If $V_{\lambda}$ is finite-dimensional with dimension $d_{\lambda}$, then $M_{\lambda}(z)$ is  a $d_{\lambda} \times d_{\lambda}$ matrix with entries in $U(\mathfrak{g})$. Now for any $\nu\in \mathcal{P}_{+}$, we consider the map
	$$\delta_{\lambda,\nu}: U(\frg) \rightarrow \End V_\lambda\otimes  \End V_\nu$$
	defined by
	$$ \delta_{\lambda,\nu}(x)= \pi_{\lambda}(x)\otimes \textrm{id} + \textrm{id} \otimes \pi_\nu(x) \text{ for } x\in \frg,$$
	which is extended to be a homomorphism of associative algebras.  Then we define
	\begin{equation}
		M_{\lambda,\nu}(z) = \delta_{\lambda,\nu}(z) - \pi_{\lambda}(z) \otimes \textrm{id} -  \textrm{id} \otimes \pi_{\nu}(z) \in R_{\lambda,\nu}(\frg). 
	\end{equation}
	It is known from [Ko1] that the infinitesimal characters occurring in the tensor product $V_{\lambda}\otimes V_{\nu}$ are of the form $\chi_{\nu+\lambda_{i}}$, where $\{\lambda_1, \cdots, \lambda_k\}$ denotes the set of distinct weights of $V_{\lambda}$. The eigenvalues of $M_{\lambda,\nu}(z)$ on the space 
	$V_{\lambda}\otimes V_{\nu}$ are therefore of the form 
	$$\chi_{\nu+\lambda_i}(z)-\chi_{\lambda}(z)-\chi_{\nu}(z).$$

	Another kind of elements in $(\End V_\lambda\otimes \mathscr{P}(\frg))^{\frg}$, which was called 
	M-type elements defined in [Ki1,Ki2]. We construct  a similar kind of  elements in $R_{\lambda}(\mathfrak{g})$ as follows.
	Assume $\{x_1,\cdots, x_m\}$ is a basis of $\frg$ and 
	$\{x_1^{*},\cdots,x_m^{*} \}$ is its dual basis relative to the Killing form $B(.,.)$. Write $$[x_i, x_j]=\sum\limits_{s=1}^{m}c_{i,j}^{s}x_s, 
	[x_i, x_j^{*}]=\sum\limits_{s=1}^{m}d_{i,j}^{s}x_s^{*},
	\ c_{i,j}^{s},  d_{i,j}^{s} \in \mathbb{C}, i,j =1,\cdots , m. $$
	For $x \in \frg$, define the substitution operator $$i_x:U(\frg)\rightarrow U(\frg)$$
	$$i_{x}(y) = B(x,y), \  \ \forall \ y \in \frg,$$
	$$i_x(u_1 \cdots u_k)=\sum\limits_{j=1}^{k}i_x(u_j)u_1\cdots u_{j-1}u_{j+1}\cdots u_k, \  \ \forall \  u_i \in \frg.$$
	\begin{prop}For any  $z\in Z(\frg)$, such element
		\begin{equation}
			M_{z}(\lambda)=\sum\limits_{i=1}^m \pi_{\lambda}(x_i^{*})\otimes i_{x_i}(z)
		\end{equation}
		belongs to $ R_{\lambda}(\frg). $
		
	\end{prop} 
	
	\begin{proof}
		From the ad-invariance of the Killing form, i.e. $$B([x_i, x_j], x_k^{*})+B(x_j, [x_i,x_k^{*}])=0$$ 
		for $1 \leq i, j, k \leq m$, 
		we know that $c_{i,j}^k+d_{i,k}^{j}=0$.
		It is straightforward to check 
		$$[adx_i,i_{x_j}]=\sum\limits_{s=1}^{m}c_{i,j}^{s}i_{x_s}, \ \ i,j =1,\cdots , m.$$
		Then for any  $z\in Z(\frg)$, the equation
		$$x_j. M_{z}(\lambda)=\sum\limits_{i=1}^m [\pi_{\lambda}(x_j),
		\pi_{\lambda}(x_i^{*})]\otimes i_{x_i}(z)+\pi_{\lambda}(x_i^{*})\otimes [x_j, i_{x_i}(z)]=0
		$$implies that $M_{z}(\lambda)\in  R_{\lambda}(\frg). $
	\end{proof}

	\section{Schur polynomials and  $R_{\omega_{k}}(\frh)$}
	
	Assume $\{\omega_k \ | \ k=1,\cdots, n-1\}$ is a set of 
	the fundamental dominant weights for  the special linear simple Lie algebra $\mbox{sl}_{n}(\mathbb{C}) $, which are minuscule  in the sense that the Weyl group action of $\mbox{sl}_{n}(\mathbb{C}) $ on the weight set of $V_{\omega_k}$ is transitive. Since $\omega_k$ is minuscule, it follows from Proposition 3.7 that 
	$$\psi_{\omega_k}:R_{\omega_k}(\frg)\rightarrow R_{\omega_k}(\frh)$$ is a  $J$-module isomorphism. 
	This  Section is  devoted to studying the  Weyl group module  $R_{\omega_{k}}(\frh)$.

	Let $W(A_{n-1})$ be the Weyl group  of $\mbox{sl}_{n}(\mathbb{C})$, which is identified with the symmetric group $S_n$ acting on certain inner product space with dimension $n-1$. The invariant ring $U(\frh)^{W(A_{n-1})}$ is generated by $n-1$ algebraically independent elements (2.4.1.1 in [SYS]). To construct the free basis for $R_{\omega_{k}}(\frh)$, we shall require the following results about symmetric polynomials. 
	
	The symmetric group $S_{n}$ acts on the  polynomial ring $\mathbb{C}[x_{1}, \dots, x_{n}]$ by permuting the variables and the invariant set  $\mathbb{C}[x_{1}, \dots, x_{n}]^{S_{n}}$ consisting of  symmetric polynomials forms a graded subring. There are many natural choices of bases for this ring.

	For any $n, t \in \mathbb{N}$ and $n$ algebraically independent variables $x_{1}, \dots, x_{n}$, the $t$-th power sum is defined by
	$$
	\frp_{t}(x_{1}, \dots, x_{n}) = \sum_{i = 1}^{n} x_{i}^{t}.
	$$

	Suppose $\alpha=(\alpha_{1}> \alpha_{2}> \cdots>\alpha_n \geq 0)$ is a partition of $
	|\alpha|=\sum\limits_{i=1}^{n}\alpha_i$. 
	Write $\alpha=\lambda+\delta,$
	where $\ \delta=(n-1,n-2, \cdots, 1,0)$ and $ \lambda=(\lambda_{1},\cdots,\lambda_{n})$ are two partitions.
	The nonzero $\lambda_i$ in $(\lambda_{1},\cdots,\lambda_{n})$ are called the parts of $\lambda$, and the number of parts is  the length of $\lambda$, denoted by $l(\lambda)$.
	Consider $A_{\alpha}$ obtained by antisymmetrizing the monomial $X^{\alpha}=x_1^{\alpha_{1}}\cdots x_n^{\alpha_{n}}$, i. e. 
	$$A_{\alpha}=\sum\limits_{\sigma\in S_n}\mbox{sgn}(\sigma)\sigma(X^{\alpha}), 
	\  \sigma(X^{\alpha})=x_1^{\alpha_{\sigma(1)}}\cdots x_n^{\alpha_{\sigma(n)}}.$$
	Then the symmetric polynomial $$\frs_{\lambda}(x_1,\cdots,x_n)=A_{\alpha}A_{\delta}^{-1}=A_{\lambda+\delta}A_{\delta}^{-1}$$is called the
	Schur polynomial in the variables $x_1, \cdots , x_n$ corresponding to the partition $\lambda$ with $l(\lambda)\leq n$.

	Each Schur polynomial can be expressed as a polynomial in the complete symmetric functions $h_r(x_1, \cdots , x_n)=\frs_{(r)}$ by the formula:
	$$\frs_{\lambda}(x_1,\cdots,x_n)=\mbox{det}(h_{\lambda_{i}-i+j}(x_1,\cdots,x_n))_{1\leq i, j \leq n }$$which is often referred to as the Jacobi-Trudi identity [M].

	Given a partition $\lambda$, the Young diagram of $\lambda$ consists of $l(\lambda)$ rows of adjacent squares: the i-th row has $\lambda_i$ squares, $i=1, \cdots, l(\lambda)$. The transpose of the Young diagram  $\lambda$, denoted by $\lambda'$, is the Young diagram obtained by transposing the columns and rows of $\lambda$.
	
	Set $${P}(t,n)=\{\lambda=(\lambda_{1}\geq \cdots \geq \lambda_{n}\geq 0) \  | \  l(\lambda)\leq n, |\lambda|=t \}$$It is known that the following two sets are the bases over $\mathbb{C}$  of the homogeneous symmetric polynomials with degree $t$ in $n$ variables [F]: 
	\begin{equation}
		\{\frs_{\lambda}(x_1,\cdots,x_n) \ | \ \lambda \in  {P}(t,n)\};
	\end{equation}
	\begin{equation}
		\{\frp_{\lambda'}(x_1,\cdots,x_n) =\prod_{i=1}^{t}\frp_{\lambda_{i}'}(x_1,\cdots,x_n)
		\ | \ \lambda \in  {P}(t,n), \lambda'=(\lambda_1',\cdots, \lambda_{t}')\}.
	\end{equation}
	For  $1 \leq k \leq n-k$, denote the set $$\sum_{k\times (n-k)}=\{ \lambda \ | \ l(\lambda) \leq k, \lambda_1 \leq n-k\}.$$
	The following theorem is due to [T]:
	\begin{thm}
		The 	
		Schur polynomial $$\frs_{\lambda}(x_{1}, \dots, x_{k}) \notin (\frp_{1}(x_1,\cdots,x_n), \dots, \frp_{n}(x_1,\cdots,x_n)) \cdot \mathbb{C}[x_{1}, \dots, x_{n}]$$ if and only if $\lambda \in \sum_{k\times (n-k)}$.
	\end{thm}
	
	Assume  $\{e_{1},\cdots, e_{n}\}$ is a standard basis of vector space $\mathbb{C}^{n}$.
	The  natural representation $V_{\omega_{1}}=\mathbb{C}^{n}$ is given by
	$$\pi_{\omega_{1}}(X_{ij}) e_{s} = \delta_{js}e_{i}, i \neq j, 1 \leq i, j, s \leq n;$$$$
	\pi_{\omega_{1}}(X_{ii}) e_{s} = (\delta_{is} - \frac{1}{n})e_{s}, \ 1 \leq i, s \leq n.$$For $k=1,\cdots, n-1$, the $k$-th fundamental representation with dimension $d_{\omega_{k}}$ for   $\mbox{sl}_{n}(\mathbb{C})$ is given by 
	$V_{\omega_{k}}=\wedge ^{k}
	V_{\omega_{1}}$, which is a  
	minuscule representation.

	Take an ordered basis $\{e_{i_{1}}\wedge e_{i_2}\wedge \cdots \wedge e_{i_{k-1}}\wedge e_{p} (1\leq i_1 < i_2 < \cdots < i_{k-1} < p \leq n-1),\cdots, 
	e_{j_1}\wedge e_{j_2}\wedge \cdots \wedge e_{j_{k-1}} \wedge e_{n} (1\leq j_1 < j_2 < \cdots < j_{k-1} \leq n-1)
	\}$ with respect to the lexicographical order on $\mathbb{N}^{n}$  for $V_{\omega_{k}}$.

	\begin{prop} Let  $1 \leq k \leq n-k$. Under the above ordered basis, the matrices  $$\mbox{diag}\{\frs_{\lambda}(X_{11},\cdots,
		X_{kk} 
		),\cdots,
		\frs_{\lambda}(X_{i_1,i_1},\cdots,
		X_{i_{k-1},i_{k-1}},X_{pp}
		), \cdots,$$$$	\frs_{\lambda}(X_{j_1,j_1},\cdots,X_{j_{k-1},j_{k-1}},0), \cdots, \frs_{\lambda}(X_{n-k-1,n-k-1},\cdots,
		X_{n-1,n-1},0)  \}$$ where  $\lambda$ varies in $\sum_{k\times ( n-k)}$, form  a free basis for  $J$-module 
		$	R_{\omega_{k}}(\mathfrak{h}) $.
	\end{prop}
	\begin{proof}
		
		Note  that
		$\End V_{\omega_{k}} \simeq V_{\omega_{n-k}} \otimes V_{\omega_{k}} \simeq \sum_{i = 1}^{k} V_{\omega_{i} + \omega_{n-i}}$ as  $\mbox{sl}_{n}(\mathbb{C})$-modules.
		Since $ V_{\omega_{k}}$ is minuscule, each constituent of 	
		$\End V_{\omega_{k}}$ is small by Theorem 1.1 of [R] and Proposition 3.3.
		It was observed by Kostant [Ko3] (and probably by many others as well) that the zero weight space $V_{\omega_{i} + \omega_{n-i}}^{0}$ of small module
		$V_{\omega_{i} + \omega_{n-i}}$ is simple as a $S_n$-module.
		Therefore,		
		$$	R_{\omega_{k}}(\mathfrak{h}) =  (\End_{\frh}V_{\omega_k}\otimes U(\frh))^{W(A_{n-1})}=
		(\End_{\frh}V_{\omega_k}\otimes U(\frh))^{S_n}$$$$=
		(f_{1} \dots,f_{n-1})R_{\omega_{k}}(\mathfrak{h}) \oplus \sum_{i = 1}^{k}H_{\omega_{i} + \omega_{n-i}}^{0}, $$	
		where $\{f_1, \cdots, f_{n-1}\}$ is any basic generator set for $J$ and  $H_{\omega_{i} + \omega_{n-i}}^{0}$ is a graded $m_{\omega_{i} + \omega_{n-i}}^{0}$-dimensional linear space over $\mathbb{C}$.

		From Equation (4.1), (4.2) and 
		$$
		\psi_{\omega_{k}}(M_{\mathfrak{C}_{s}}(\omega_{k}))	= \psi_{\omega_{k}}(\sum\limits_{1\leq i\neq j\leq n}\pi_{\omega_k}(X_{ij})\otimes 
		i_{X_{ij}^{*}}(\mathfrak{C}_s))+ \psi_{\omega_{k}}(\sum\limits_{i=1}^{n-1}\pi_{\omega_k}(X_{ii})\otimes 
		i_{X_{ii}^{*}}(\mathfrak{C}_s))$$$$=	
		\sum_{i=1}^{n-1}s\pi_{\omega_k}(X_{ii}) \otimes X_{ii}^{s-1}
		=s\mbox{diag}\{\frp_{s-1}(X_{11},\cdots,
		X_{kk} 
		),\cdots,$$$$
		\frp_{s-1}(X_{i_1,i_1},\cdots,
		X_{i_{k-1},i_{k-1}},X_{pp}
		),\cdots,\frp_{s-1}(X_{j_1,j_1},\cdots,X_{j_{k-1},j_{k-1}},0),$$$$
		\cdots,\frp_{s-1}(X_{n-k-1,n-k-1},\cdots,
		X_{n-1,n-1},0)
		\}-s\frac{k}{n}.\frp_{s-1}(X_{11},\cdots,
		X_{nn} 
		)I_{d_{\omega_k}} \in 	R_{\omega_{k}}(\mathfrak{h})
		$$to obtain such diagonal matrices
		$$\mbox{diag}\{\frs_{\lambda}(X_{11},\cdots,
		X_{kk} 
		),\cdots,
		\frs_{\lambda}(X_{i_1,i_1},\cdots,
		X_{i_{k-1},i_{k-1}},X_{pp}
		), \cdots,$$$$	\frs_{\lambda}(X_{j_1,j_1},\cdots,X_{j_{k-1},j_{k-1}},0), \cdots, \frs_{\lambda}(X_{n-k-1,n-k-1},\cdots,
		X_{n-1,n-1},0)  \} \in 	R_{\omega_{k}}(\mathfrak{h}). $$	
		Furthermore, these matrices are  contained in $ \sum_{i = 1}^{k} H_{\omega_{i} + \omega_{n-i}}^{0} $
		if and only if $\lambda \in \sum_{k\times (n-k)}$ by Theorem 4.3. Note that the dimension of $\sum_{i = 1}^{k}H_{\omega_{i} + \omega_{n-i}}^{0}$ over $\mathbb{C}$ is the rank of $J$-module $R_{\omega_{k}}(\mathfrak{h})$, which is  equal to $\dim \End_{\mathfrak{h}} V_{\omega_{k}} = \binom{n}{k}$. 
		The number of the set $\sum_{k\times (n-k)}$ is also equal to $\binom{n}{k}$. Thus the above  diagonal linearly independent matrices form a $\mathbb{C}$-basis of the vector space $\sum_{i = 1}^{k} H_{\omega_{i} + \omega_{n-i}}^{0}$ of dimension $\binom{n}{k}$. It is also known that the $\mathbb{C}$-basis for the vector space $ \sum_{i = 1}^{k} H_{\omega_{i} + \omega_{n-i}}^{0}$ is a  free basis for $J$-module
		$	R_{\omega_{k}}(\mathfrak{h}) $. Therefore, the Proposition 4.4 follows.
	\end{proof}

	\section{Basis Structures for  $R_{\omega_k}(\mbox{sl}_{n}(\mathbb{C}))$ and 
		$R_{\omega_k,\nu}(\mbox{sl}_{n}(\mathbb{C}))$ }

	In this section, we will construct the basis for  
	$R_{\omega_{k}}(\mbox{sl}_{n}(\mathbb{C}))$ and $R_{\omega_{k},\nu}(\mbox{sl}_{n}(\mathbb{C}))$ based on Proposition 3.7 and 3.9.

	Let $\mbox{Gr}^{k}(\mathbb{C}^{n})$
	denote the  Grassmannian variety of subspaces in $\mathbb{C}^{n}$ with dimension $k$. It follows from [F] that the complex coefficient cohomology ring of 
	Grassmannian variety 
	$\mbox{Gr}^{k}(\mathbb{C}^{n})$ is isomorphic to the following associative  algebraic structure  $L(k,n-k)$ given by Theorem 1 in  [Hi] and Theorem 3.1 in [CL], which is stated as follows:
	
	\begin{thm}Let $1 \leq k \leq n-k$. Suppose that $L(k,n-k)$ denotes the graded $\mathbb{C}$-algebra $\mathbb{C}[w_1,\cdots,w_k,z_1, \cdots, z_{n-k}]/I(k,n-k)$, where
		$\mbox{deg}(w_i)=i=\mbox{deg}(z_i)$ and 	
		$I(k,n-k)$ is generated by the single non-homogeneous relation
		$$(1+w_1+\cdots +w_k)(1+z_1+\cdots+z_{n-k})=1.$$Then $L(k,n-k)$ is isomorphic to 
		$$\mathbb{C}[w_1,\cdots,w_k]/J(k,n-k),$$ where $J(k,n-k)$ is generated by	 
		$k$ homogeneous relations $f_{1,n-k},\cdots,f_{k,n-k}$ given  by 
		\begin{equation}
			f_{s,n-k}=\sum_{\vec{t} = (t_{1}, t_{2}, ..., t_{k})}\frac{t_s+\cdots +t_{k}}{t_1+\cdots +t_k}\binom{t_1+\cdots +t_k}{t_1,\cdots,t_k}w_1^{t_1}\cdots w_k^{t_k}.
		\end{equation}
	    The summation is over non-negative integers satisfying $t_1+2t_2+\cdots+kt_k=n-k+s$,
		$ s=1, \cdots k$.
	\end{thm}

	\begin{cor}	Assume $\omega_{k}$ is totally subordinate to $\nu \in \mathcal{P}_{+}$. Then the following associative  algebraic  structures are  isomorphic:$$	R_{\omega_{k}}(\mathfrak{g})/
		(f_{1} \dots,f_{n-1})R_{\omega_{k}}(\mathfrak{g})\simeq	R_{\omega_{k}}(\mathfrak{h})/
		(f_{1} \dots,f_{n-1})R_{\omega_{k}}(\mathfrak{h})\simeq R_{\omega_k,\nu}(\frg)\simeq L(k,n-k), $$
		where $\{f_1, \cdots, f_{n-1}\}$ is any basic generator set for $J$.
		Besides, they  are all isomorphic to the complex coefficient cohomology ring of 
		Grassmannian variety 
		$\mbox{Gr}^{k}(\mathbb{C}^{n})$.
	\end{cor}
	\begin{proof}
		It is known that 
		$\mbox{Gr}^{k}(\mathbb{C}^{n})$ is isomorphic to the complex linear space with the basis 
		$$\{\frs_{\lambda}(x_1,\cdots, x_n) \ | \ \lambda \in \sum_{k\times (n-k)}  \}$$ with dimension $\binom{n}{k}$ [F], which is also isomorphic to the associative algebra $L(k,n-k)$ by Theorem 5.1.

		Since $V_{\omega_{k}}$ is minuscule, the restriction map $\psi_{\omega_{k}}$ is an associative algebra homomorphism and also a  $J$-module isomorphism. Therefore, by Kostant's decomposition theorem for the universal enveloping algebra, we have the following associative algebraic isomorphisms: $$R_{\omega_{k}}(\mathfrak{g})/
		(f_{1} \dots,f_{n-1})R_{\omega_{k}}(\mathfrak{g})\simeq 
		R_{\omega_{k}}(\mathfrak{h})/
		(f_{1} \dots,f_{n-1})R_{\omega_{k}}(\mathfrak{h})\simeq		
		\sum_{i = 1}^{k} H_{\omega_{i} + \omega_{n-i}}^{0}.$$	
		From Proposition 4.4, we also know that $\sum_{i = 1}^{k} H_{\omega_{i} + \omega_{n-i}}^{0}$ is isomorphic to the complex linear space with the basis 
		$\{\frs_{\lambda}(x_1,\cdots, x_k) \ | \ \lambda \in \sum_{k\times (n-k)}  \}$. Therefore,
		$$	R_{\omega_{k}}(\mathfrak{g})/
		(f_{1} \dots,f_{n-1})R_{\omega_{k}}(\mathfrak{g})\simeq	R_{\omega_{k}}(\mathfrak{h})/
		(f_{1} \dots,f_{n-1})R_{\omega_{k}}(\mathfrak{h})\simeq L(k,n-k). $$		 
		On the other hand, if $s_i$ is a free basis for $Z(\frg)$-module 
		$R_{\omega_{k}}(\mathfrak{g})$, then $\gamma_{\nu}(s_i)$ is a basis of $R_{\omega_{k},\nu}(\mathfrak{g}) $ as a $\mathbb{C}$-vector space by Proposition 3.9.
		Hence, $R_{\omega_k,\nu}(\frg) \simeq	R_{\omega_{k}}(\mathfrak{g})/
		(f_{1} \dots,f_{n-1})R_{\omega_{k}}(\mathfrak{g})\simeq L(k,n-k)$ as well.
	\end{proof}

	\begin{thm}
		The set of matrices $M_{\omega_{k}}(\mathfrak{C}_{2}),\dots, M_{\omega_{k}}(\mathfrak{C}_{k+1})$
		is a minimal set of generators of  $R_{\omega_{k}}(\mbox{sl}_{n}(\mathbb{C}))$ over $Z(\mbox{sl}_{n}(\mathbb{C}))$.
	\end{thm}
	\begin{proof}
		For $t \in \mathbb{N}$, assume that $R_{\omega_k}^{t}(\frh)$ is the set of all matrices in $\mathfrak{h}$-endomorphism algebra $R_{\omega_k}(\frh)$ with entries in  $U(\mathfrak{h})^{t}$, which consists of all polynomials with degree less than or equal to $t$.
		From Proposition 3.6, we know that $$
		\psi_{\omega_{k}}(M_{\omega_{k}}(\mathfrak{C}_s))
		=	
		\sum_{i=1}^{n}s\pi_{\omega_k}(X_{ii}) \otimes X_{ii}^{s-1}	\mbox{mod} (R_{\omega_k}^{s-2}(\frh))$$$$
		=s\mbox{diag}\{\frp_{s-1}(X_{11},\cdots,
		X_{kk} 
		),\cdots,
		\frp_{s-1}(X_{i_1,i_1},\cdots,
		X_{i_{k-1},i_{k-1}},X_{pp}
		),\cdots,$$$$\frp_{s-1}(X_{j_1,j_1},\cdots,X_{j_{k-1},j_{k-1}},X_{n,n}),
		\cdots,\frp_{s-1}(X_{n-k-1,n-k-1},\cdots,
		X_{n-1,n-1},X_{nn})
		\}$$$$-s\frac{k}{n}.\frp_{s-1}(X_{11},\cdots,
		X_{nn} 
		)I_{d_{\omega_k}}
		\mbox{mod} (R_{\omega_k}^{s-2}(\frh)).
		$$Note that the notations $M_{\omega_{k}}(\mathfrak{C}_s)$ and $M_{\mathfrak{C}_{s}}(\omega_{k})$ are different, which are given by Equations (3.10) and (3.13) respectively. It follows from the above expression of $\psi_{\omega_{k}}(M_{\omega_{k}}(\mathfrak{C}_s))$ and Equation (4.2) that 	
		$\{\psi_{\omega_{k}}(M_{\omega_{k}}(\mathfrak{C}_s) \ | \ s=2, \cdots, k+1\}$	
		is  a minimal  
		set of generators for  $R_{\omega_{k}}(\mathfrak{h})$  over $J$.
		
		Since $V_{\omega_{k}}$ is minuscule, the restriction map $\psi_{\omega_{k}}$ is an associative algebra homomorphism and also a  $J$-module isomorphism. Hence,	$\{M_{\omega_{k}}(\mathfrak{C}_s) \ | \ s=2, \cdots, k+1\}$	
		is   a minimal  
		set of generators for  $R_{\omega_{k}}(\mathfrak{g})$  over $Z(\mbox{sl}_{n}(\mathbb{C}))$.
		More precisely, the set $$\{M_{\omega_{k}}(\mathfrak{C}_{2})^{s_1}\dots M_{\omega_{k}}(\mathfrak{C}_{k+1})^{s_k} \ | \  (s_1 \geq s_2 \cdots \geq s_k)\in 
		\sum_{k\times (n-k)}	
		\}$$
		is a set of the free basis for $Z(\mbox{sl}_{n}(\mathbb{C}))$-module $R_{\omega_{k}}(\mbox{sl}_{n}(\mathbb{C}))$.	
	\end{proof}

	Assume the finite-dimensional complex vector space $V$ is the direct sum of a family of subspaces of $V_i$, written by $$V=\bigoplus_{i=1}^s V_i.$$ 
	For $i=1,\cdots, s$, the $i$-th projection operator is defined by 
	$$P_i:V \rightarrow V_i, \ \sum_{j = 1}^{s}v_j \mapsto v_i, \ \forall \ v_j \in V_j.$$	
	\begin{defi}
		If $D_i:V  \rightarrow  V$ is a series of diagonalizable linear operators with special resolutions, i. e.
		$$D_i=\sum_{j = 1}^{s}a_{ji}P_j, i=1,\cdots, s.$$
		For any pair $V_i, V_j$ ($i\neq j$), if there exists some $D_k$ so that $D_k|_{V_i}
		\neq D_k|_{V_j}
		$, i.e. $a_{ik}\neq a_{jk}$ for some $k$, then we say $V$ is separated by $D_1,\cdots, D_s$.

	\end{defi}

	\begin{lemma}
		Suppose  $P$ is the complex  vector space spanned by the  projection operators $P_1, \cdots, P_s$. If $D_1, \cdots, D_s$ is a basis of the vector space $P$, then  $V$ is separated by $D_1,\cdots, D_s$.
	\end{lemma}
	\begin{proof}Write $$[D_1, \cdots, D_s]=[P_1, \cdots, P_s]A, \  \ A=(a_{ij})_{i,j=1}^{s}.$$Since  the  projection operators $P_1, \cdots, P_s$ are linear independent,  the matrix $A$ is invertible. Any two distinct rows of $A$ are linearly independent. Thus there exists the
	$ k$-th column such that $a_{ik}\neq a_{jk}$ for each pair $i\neq j$.
	\end{proof}

	Let
	\begin{equation}
		\lambda_{i_{1} i_{2} \cdots i_{k}}= (-\omega_{i_{1}-1} + \omega_{i_{1}}) + \dots + (-\omega_{i_{k}-1} + \omega_{i_{k}}),\  1 \leq i_{1} < \dots < i_{k} \leq n
	\end{equation}denote the weight of $e_{i_1}\wedge \cdots \wedge e_{i_k}$ for $\mbox{sl}_{n}(\mathbb{C})$-module $V_{\omega_k}$ and $\nu \in 
	\mathcal{P}_{+}$. 
	The tensor module decomposition is
	\begin{equation}
		V_{\omega_{k}} \otimes V_{\nu} = \sum_{\substack{1 \leq i_{1} < i_{2} < \dots < i_{k} \leq n  \\ \nu + \lambda_{i_{1} i_{2} \dots i_{k}} \in \mathcal{P}_{+}}} V_{\nu +  \lambda_{i_{1} i_{2} \dots i_{k}}}.
	\end{equation}
	Recall 
	that $M_{\omega_{k}, \nu}(\mathfrak{C}_{p}), p=2, \cdots,n,$ is a series  of diagonalizable linear operators on
	the tensor $\mbox{sl}_{n}(\mathbb{C})$-module
	$V_{\omega_{k}} \otimes V_{\nu} $, which acts on each constituent  $V_{\nu +  \lambda_{i_{1} i_{2} \dots i_{k}}}$
	with the eigenvalue
	\begin{equation}
		f_{\mathfrak{C}_{p}, i_{1}i_{2} \dots i_{k}}(\nu)=\chi_{\nu+\lambda_{i_{1} i_{2} \dots i_{k}}}(\mathfrak{C}_p)-\chi_{\nu}(\mathfrak{C}_p)-\chi_{\omega_k}(\mathfrak{C}_p).
	\end{equation}

	\begin{prop}
	Assume $\omega_{k}$ is totally subordinate to $\nu \in \mathcal{P}_{+}$.	
	Define	$$t_0(\nu)=\mbox{min}\{t \in \mathbb{N} \ | \ \forall \  \lambda_{i_{1} i_{2} \dots i_{k}}
	\neq  \lambda_{j_{1} j_{2} \dots j_{k}}, \exists  \ t \in \{2, \cdots, n\},   \  \mbox{such that}$$$$	M_{\omega_{k}, \nu}(\mathfrak{C}_{s})
	|_{V_{\nu + \lambda_{i_{1} i_{2} \dots i_{k}}}} \neq  	M_{\omega_{k}, \nu}(\mathfrak{C}_{s}) |_{V_{\nu + \lambda_{j_{1} j_{2} \dots j_{k}}}} \ \mbox{for} \ s=2,3, \cdots, t\}.
	$$Then the integer $t_0(\nu) = k+1$. 
\end{prop}

	\begin{proof}Since $\omega_{k}$ is totally subordinate to $\nu \in \mathcal{P}_{+}$, it follows from  Theorem 5.4 that the finite set of  matrices $$I_{\omega_{k}, \nu}=\{
		M_{\omega_{k}, \nu}(\mathfrak{C}_{2})^{s_1}\cdots
		M_{\omega_{k}, \nu}(\mathfrak{C}_{k+1})^{s_k} \ | \  (s_1 \geq s_2 \cdots \geq s_k)\in 
		\sum_{k\times (n-k)}	\}$$
		is a basis of the $\mathbb{C}$-vector space $R_{\omega_{k}, \nu}(\mbox{sl}_{n}(\mathbb{C}))$.
		
		Note that  the set of projections of $\mathfrak{g}$-modules $$
		\{P_{i_{1} i_{2} \dots i_{k}}:
		V_{\omega_{k}} \otimes V_{\nu}\rightarrow V_{\nu + \lambda_{i_{1} i_{2} \dots i_{k}}} \ | \
		1 \leq i_1 < i_2 < \dots < i_k \leq n, \  \nu + \lambda_{i_{1} i_{2} \dots i_{k}} \in 	\mathcal{P}_{+}
		\}$$
		is also a linear basis of $R_{\omega_{k}, \nu}(\mbox{sl}_{n}(\mathbb{C}))$, which is isomorphic to the associative algebra $L(k,n-k)$, whose generator $w_i$ corresponds to $	M_{\omega_{k}, \nu}(\mathfrak{C}_{i+1})$, $i=1,\cdots,k$.
		By Lemma 5.6,  any distinct pairs $V_{\nu + \lambda_{i_{1} i_{2} \dots i_{k}}}$ and
		$V_{\nu + \lambda_{j_{1} j_{2} \dots j_{k}}}$ can be separated by certain operator
		$$	M_{\omega_{k}, \nu}(\mathfrak{C}_{2})^{s_1}\cdots
		M_{\omega_{k}, \nu}(\mathfrak{C}_{k+1})^{s_k}, \ (s_1 \geq s_2 \cdots \geq s_k)\in 
		\sum_{k\times (n-k)}
		$$in the basis set $I_{\omega_{k}, \nu}$. Thus 
		there exists an  operator
		$
		M_{\omega_{k}, \nu}(\mathfrak{C}_{s})$  which has different eigenvalues acting on 
		$V_{\nu + \lambda_{i_{1} i_{2} \dots i_{k}}}$ and
		$V_{\nu + \lambda_{j_{1} j_{2} \dots j_{k}}}$ for $s \in \{2, \cdots, k+1\}$.  Hence, $t_0(\nu) \leq k+1$.

		Since $R_{\omega_{k}, \nu}(\mbox{sl}_{n}(\mathbb{C}))\simeq L(k,n-k)$, the algebraic structure $L(k,n-k)$ in  Theorem 5.1 implies that such elements $w_{1}^{s_1}\cdots w_{k}^{s_{k}}$ in $L(k,n-k)$, with $s_{1}+2s_{2}+\cdots +ks_{k} \leq n-k$, are linearly independent.
		A contradiction arises. 
		Hence we claim that the integer $t_0(\nu) = k+1$.

	\end{proof}

	\section{$R_{\omega_k,\nu}(\mbox{sl}_{n}(\mathbb{C}))$ and the PTE problem}

	In this Section, we will 
	introduce the relationship between the  basis structure of  $R_{\omega_{k},\nu}(\mbox{sl}_{n}(\mathbb{C}))$  and the PTE problem. And we will finally  prove Wright's conjecture holds true.
	We always assume $1 \leq k \leq n-k$.
	
	\begin{prop} Assume $1 \leq i_{1} < i_{2} < \dots < i_{k} \leq n$, $1 \leq j_{1} < j_{2} < \dots < j_{k} \leq n$ and $i_{1} \leq j_{1}$.  Let $r$ denote the number of elements in the set $\{i_1, \cdots, i_k\}\cap \{j_1,\cdots, j_k\}$. 
		For  $\nu \in \mathcal{P}_{+}$ satisfying $\nu +  \lambda_{i_{1} i_{2} \dots i_{k}}, \nu +  \lambda_{j_{1} j_{2} \dots j_{k}} \in \mathcal{P}_{+}$, if $f_{\mathfrak{C}_{p}, i_{1} i_{2} \dots i_{k}}(\nu) = f_{\mathfrak{C}_{p}, j_{1} j_{2} \dots j_{k}}(\nu)$ for $2 \leq p \leq s+1$, then the PTE problem  with size $k-r$ and degree $s$ has a non-trivial solution.
	\end{prop}
	\begin{proof}
		Let the weight $\nu = a_{1}\omega_{1} + a_{2}\omega_{2} + \dots + a_{n-1}\omega_{n-1} $ such that $\nu + \lambda_{i_{1} i_{2} \dots i_{k}}, \nu + \lambda_{j_{1} j_{2} \dots j_{k}} \in \mathcal{P}_{+}$. Denote $$[f(\nu)] = [f_{1}(\nu), \dots, f_{n-1}(\nu), f_{n}(\nu)]$$ with $ f_{1}(\nu) \geq \dots \geq f_{n-1}(\nu) \geq f_{n}(\nu) =0$, which is the corresponding Young pattern with $f_i(\nu) =\sum\limits_{j=i}^{n-1}a_j, i= 1, 2, ..., n-1$. 
		Recall that $\lambda_{i_{1} i_{2} \cdots i_{k}}$ is given by Equation (5.7).
		Then we get each component of Young pattern $[f(\nu + \lambda_{i_{1} i_{2} \cdots i_{k}})]$ as follows: 
		\begin{equation*} 
			f_{t}(\nu + \lambda_{i_{1} i_{2} \cdots i_{k}}) = 
			\begin{cases}
				0 & \text{ if } t = n , \\
				f_{t}(\nu) + 1 & \text{ if } t \in \{ i_{1}, i_{2}, \dots, i_{k}\}, \\
				f_{t}(\nu) & \text{otherwise}.
			\end{cases}
		\end{equation*}
		Define
		$$a_{0} = \dfrac{\sum\limits_{j=1}^{n}f_{j}(\nu + \lambda_{i_{1} i_{2} \cdots i_{k}}) }{n},\ 	c_{t} = a_{0} - n+t, t = 0, 1, 2, \dots, n.
		$$
		Let $\lambda_{j_{1} j_{2} \cdots j_{k}}$ be another weight of $V_{\omega_{k}}$ and 
		$$I=\{i_{1}, \dots, i_{k}\}, \ J=\{j_{1}, \cdots, j_{k}\}, \ K = \{1,\cdots,n\}-I-J,$$$$	
		\{i_{\alpha_{1}}, \dots, i_{\alpha_{k-r}} \} = \{i_{1}, \dots, i_{k}\} \backslash \{j_{1}, \dots, j_{k}\}, \quad i_{\alpha_{1}} < \dots <  i_{\alpha_{k-r}},$$
		$$\{j_{\beta_{1}}, \dots, j_{\beta_{k-r}}\} = \{j_{1}, \dots, j_{k}\} \backslash \{i_{1}, \dots, i_{k}\}, \quad j_{\beta_{1}} < \dots < j_{\beta_{k-r}}.$$
		If $M_{\omega_{k}, \nu}(\mathfrak{C}_{2})$ acts on  $V_{\nu + \lambda_{i_{1} i_{2} \dots i_{k}}}$ and $V_{\nu + \lambda_{j_{1} j_{2} \dots j_{k}}}$ with  the same eigenvalues, then by Equation (5.10) we have  $$ 0=f_{\mathfrak{C}_{2}, i_{1} i_{2} \dots i_{k}}(\nu) - f_{\mathfrak{C}_{2}, j_{1} j_{2} \dots j_{k}}(\nu)$$
		$$=\chi_{\nu + \lambda_{i_{1} i_{2} \dots i_{k}}}(\mathfrak{C}_2)-
		\chi_{\nu + \lambda_{j_{1} j_{2} \dots j_{k}}}(\mathfrak{C}_2)$$$$
		= \sum_{s = 1}^{n} (f_{s}(\nu + \lambda_{i_{1} i_{2} \cdots i_{k}}) - c_{s})^{2} - \sum_{s = 1}^{n} (f_{s}(\nu + \lambda_{j_{1} j_{2} \cdots j_{k}}) - c_{s})^{2}$$$$= (\sum_{t = 1}^{k}(f_{i_{t}}(\nu) + 1 - c_{i_{t}})^2 + \sum_{t = 1}^{k}(f_{j_{t}}(\nu) - c_{j_{t}})^2 - \sum_{t \in \{i_{1}, i_{2}, \dots, i_{k}\} \cap \{j_{1}, j_{2}, \dots, j_{k}\}}(f_{t}(\nu) - c_{t})^2) $$$$ \quad - (\sum_{t = 1}^{k} (f_{j_{t}}(\nu) + 1 - c_{j_{t}})^2 + \sum_{t = 1}^{k}(f_{i_{t}}(\nu) - c_{i_{t}})^2 - \sum_{t \in \{i_{1}, i_{2}, \dots, i_{k}\} \cap \{j_{1}, j_{2}, \dots, j_{k}\}}(f_{t}(\nu) - c_{t})^2)$$$$= (\sum_{t = 1}^{k} (f_{i_{t}}(\nu) + 1 - c_{i_{t}})^2  - \sum_{t = 1}^{k} (f_{i_{t}}(\nu) - c_{i_{t}})^2) - (\sum_{t = 1}^{k} (f_{j_{t}}(\nu) + 1 - c_{j_{t}})^2  - \sum_{t = 1}^{k} (f_{j_{t}}(\nu) - c_{j_{t}})^2)$$
		$$=2\sum_{t = 1}^{k} [(f_{i_{t}}(\nu) - c_{i_{t}}) - (f_{j_{t}}(\nu) - c_{j_{t}})]
		=2\sum_{t = 1}^{k-r}[(f_{i_{\alpha_{t}}}(\nu) - c_{i_{\alpha_{t}}}) - (f_{j_{\beta_{t}}}(\nu) - c_{j_{\beta_{t}}})].$$
		Inductively, since $M_{\omega_{k}}(\mathfrak{C}_{p})$ acting on $V_{\nu +  \lambda_{i_{1} i_{2} \dots i_{k}}}$ and $V_{\nu +  \lambda_{j_{1} j_{2} \dots j_{k}}}$ has the same eigenvalue for $3 \leq p \leq s+1$, then
		$$0=f_{\mathfrak{C}_{p}, i_{1} i_{2} \dots i_{k}}(\nu) - f_{\mathfrak{C}_{p}, j_{1} j_{2} \dots j_{k}}(\nu)=\chi_{\nu + \lambda_{i_{1} i_{2} \dots i_{k}}}(\mathfrak{C}_p)-
		\chi_{\nu + \lambda_{j_{1} j_{2} \dots j_{k}}}(\mathfrak{C}_p)
		$$$$= S_{p}(\nu + \lambda_{i_{1} i_{2} \dots i_{k}}) - S_{p}(\nu + \lambda_{j_{1} j_{2} \dots j_{k}})= \sum_{q = 1}^{p-1} \binom{p}{q} (\sum_{t = 1}^{k} (f_{i_{t}}(\nu) - c_{i_{t}})^{q} - (f_{j_{t}}(\nu) - c_{j_{t}})^{q})$$
		$$= \sum_{q = 1}^{p-1} \binom{p}{q} (\sum_{t = 1}^{k-r} (f_{i_{\alpha_{t}}}(\nu) - c_{i_{\alpha_{t}}})^{q} - (f_{j_{\beta_{t}}}(\nu)- c_{j_{\beta_{t}}})^{q}) $$
		by Equations (2.4), (2.5) and (2.6).
		Let 
		$$x_{t} = f_{i_{\alpha_{t}}}(\nu) - i_{\alpha_{t}}, \quad y_{t} = f_{j_{\beta_{t}}}(\nu) - j_{\beta_{t}} $$
		for $t = 1, 2, \dots, k-r$. Then we have 
		$$x_{t}, y_{t} \in \mathbb{Z}, \quad \{ x_{1}, \dots, x_{k-r} \} \ne \{ y_{1}, \dots, y_{k-r} \},$$
		$$x_{1} > x_{2} > \dots > x_{k-r},\quad y_{1} > y_{2} > \dots > y_{k-r},$$
		and they satisfy PTE problem
	    \begin{equation}\label{TEP2}
	    	\sum_{i = 1}^{k-r}{x_{i}^j} = \sum_{i = 1}^{k-r}y_{i}^{j}, \qquad j = 1, 2, \dots, s.
	    \end{equation}
		It means that $[x_{1}, x_{2}, \dots, x_{k-r}]$ and $[y_{1}, y_{2}, \dots, y_{k-r}]$ is a non-trivial solution of PTE problem \eqref{TEP2}. 
	\end{proof}

	\begin{cor} Assume $1 \leq k \leq n-k$.
		There exists some $\nu \in \mathcal{P}_{+}$ such that the tensor module $V_{\omega_{k}} \otimes V_{\nu}$ can be separated by operators $M_{\omega_{k}, \nu}(\mathfrak{C}_{2}),$ $\dots, M_{\omega_{k}, \nu}(\mathfrak{C}_{k+1})$. 
	\end{cor}

	\begin{proof}
		Since $k \leq n-k$, we know that  there exist two indices $I=\{i_{1}, i_{2}, \cdots, i_{k}\}$ and $J=\{j_{1}, j_{2}, \cdots, j_{k}\}$ such that $I\cap J=\emptyset$. For $\nu \in \mathcal{P}_{+}$, assume 
		$V_{\nu + \lambda_{i_{1} i_{2} \dots i_{k}}}$ and $V_{\nu + \lambda_{j_{1} j_{2} \dots j_{k}}}$ are two distinct irreducible components of $V_{\omega_{k}} \otimes V_{\nu}$. If  $ f_{\mathfrak{C}_{p}, i_{1} i_{2} \dots i_{k}}(\nu) = f_{\mathfrak{C}_{p}, j_{1} j_{2} \dots j_{k}}(\nu)$ for $2 \leq p \leq k+1$, then the  PTE problem  with size $k$ and degree $k$ has a non-trivial solution by Proposition 6.1. But it is impossible
		by Theorem 2.9.
	\end{proof}

	\begin{cor}
		Assume there exists some weight $\nu \in \mathcal{P}_{+}$ such that  $\nu +  \lambda_{i_{1} i_{2} \dots i_{k}}$ and $\nu +  \lambda_{j_{1} j_{2} \dots j_{k}}$ are both in $\mathcal{P}_{+}$. If $ f_{\mathfrak{C}_{p}, i_{1} i_{2} \dots i_{k}}(\nu) = f_{\mathfrak{C}_{p}, j_{1} j_{2} \dots j_{k}}(\nu)$ for $2 \leq p \leq k$, then $r = 0$. 
	\end{cor}
	
	\begin{proof}
		We also use the notations in the proof of Proposition 6.1. Then the  PTE problem with size $k - r$ and degree $k-1$ have a non-trivial solution. Therefore, $k-r \geq k-1+1$ by Theorem 2.9, i.e. $r = 0$.
	\end{proof}

	\begin{thm}
		An ideal solution of the PTE Problem with any degree $k$ always exists.
	\end{thm}

	\begin{proof}For a fixed positive integer $k$, there always  exists $n$ such that $1 \leq k \leq n-k$. And there  there always  exists weight 
		 $\nu \in \mathcal{P}_{+}$ such that $\omega_{k}$ is totally subordinate to  $\nu$. By Proposition 5.10,  $t_0(\nu)=k+1$ is the smallest integer $t$  such that $\mbox{sl}_{n}(\mathbb{C})$-module $V_{\omega_{k}} \otimes V_{\nu}$ can be separated  by $M_{\omega_{k}, \nu}(\mathfrak{C}_{p})$ for $2 \leq p \leq t$.
		Therefore, 		
		 the operators $M_{\omega_{k}, \nu}(\mathfrak{C}_{p})$ for $2 \leq p \leq k$ can not separate the tensor module  $V_{\omega_{k}} \otimes V_{\nu}$, i. e. there exists certain weights  $\nu +  \lambda_{i_{1} i_{2} \dots i_{k}}$ and $\nu +  \lambda_{j_{1} j_{2} \dots j_{k}}$ both in $\mathcal{P}_{+}$ such that 
		$ f_{\mathfrak{C}_{p}, i_{1} i_{2} \dots i_{k}}(\nu) = f_{\mathfrak{C}_{p}, j_{1} j_{2} \dots j_{k}}(\nu)$ for $2 \leq p \leq k$, then by Corollary 6.4, $r = 0$. Furthermore,  Proposition 6.1 implies the PTE problem with size k and degree $k-1$ has a non-trivial solution.		
		 Hence, Wright's Conjecture, i. e. the second part of Theorem 1.3, holds true.
		
	\end{proof}

\end{document}